\documentclass[11pt]{article}
\usepackage[russian]{babel}
\usepackage{graphicx}
\usepackage{amsmath}
\usepackage{fullpage}
\usepackage{amsfonts}
\usepackage{amssymb}

\newtheorem{theorem}{Theorem}

\newtheorem{corollary}[theorem]{Corollary}

\newenvironment{proof}[1][Proof]{\textbf{#1.} }{\ \rule{0.5em}{0.5em}}

\numberwithin{equation}{section}
\numberwithin{theorem}{section}

\begin{document}


\title{ Invariant subspaces for commuting operators in a real Banach space}

\author{ Victor Lomonosov and Victor Shulman}

\maketitle

\medskip

\medskip

{\bf Abstract.}

\medskip

It is proved that a commutative algebra $A$ of operators in a reflexive real Banach space has an invariant subspace if each operator $T\in A$ satisfies the condition
$$\|1- \varepsilon T^2\|_e \le 1 + o(\varepsilon) \text{ when  } \varepsilon\searrow 0,$$
where $\|\cdot\|_e$ is the essential norm.
This implies the existence of an invariant subspace for every commutative family of essentially selfadjoint operators in a real Hilbert space.

\medskip

\medskip

\section{Introduction }

One of the most known unsolved problems in the invariant subspace theory is the question of existence of a (non-trivial, closed) invariant subspace for an operator $T$ with compact imaginary part (= essentially selfadjoint operator = compact perturbation of a selfadjoint operator). There is a lot of papers concerning this subject; we only mention that the answer is affirmative for perturbations by operators from Shatten - von Neumann class $\mathfrak{S}_p$ (Livshits \cite{MSL} for $p=1$, Sahnovich    \cite{Sah} for $p=2$, Gohberg and Krein  \cite{GK}, Macaev  \cite{M1}, Schwartz \cite{Sch} --- for arbitrary $p$), or, more generally, from the Macaev ideal (Macaev  \cite{M2}). But the general question is still open.

It was proved in \cite{Lom2} that an essentially self-adjoint operator in a complex Hilbert space has an invariant real subspace. Then in \cite{Lom3} the following  general theorem of  Burnside type was proved:
\begin{theorem}\label{L1}
   Suppose that an algebra $A$  of bounded operators in a (real or complex) Banach space  $X$  is not dense in the algebra $\mathcal{B}(X)$ of all bounded operators on $X$ with respect to the weak operator topology (WOT). Then there are non-zero $x\in X^{**}, y\in X^*$, such that
\begin{equation}\label{ineq}
|(x,T^*y)| \le \|T\|_e \text{ for all } T\in A,
\end{equation}
 where $\|T\|_e$ is the essential norm of $~T$, that is $\|T\|_e = \inf\{\|T+K\|: K\in \mathcal{K}(X)\}$. Here $\mathcal{K}(X)$ is the ideal of all compact operators in $X$.
\end{theorem}

Using this result and developing a special variational techniques, Simonic \cite{Sc} has obtained a significant progress in the topic: he proved that each essentially selfadjoint operator in a real Hilbert space has invariant subspace. Deep results based on Theorem  \ref{L1} were proved then by Atzmon \cite{A}, Atzmon, Godefroy  and Kalton  \cite{AGK}, Grivaux \cite{SofGr} and other mathematicians. Here we show that every commutative family of essentially selfadjoint operators in a real Hilbert space has an invariant subspace, and consider some analogs of this result for operators in Banach spaces. Our proof is very simple and short but it heavily depends on Theorem \ref{L1}. More precisely, we use the following improvement of Theorem \ref{L1} obtained in \cite{Lom4}:

\begin{theorem}\label{L4}
 If an algebra $A\subset \mathcal{B}(X)$   is not WOT-dense in $\mathcal{B}(X)$, then there are non-zero vectors  $x\in X^{**}, y\in X^*$, such that $(x,y) \geq 0$ and
\begin{equation}\label{ineq2}
|(x,T^*y)| \le \|T\|_e (x,y), \text{ for all } T\in A.
\end{equation}
\end{theorem}

 Let us mention that if $A$ contains a non-zero compact operator then by \cite{Lom0} $A$ has a non-trivial invariant subspace $L$ in  $X$ (so $x$ can be chosen from $L$ and $y$ can be chosen from $L^\bot$).

 The original proof of Theorem \ref{L1} was essentially simplified by  Lindstrom  and Shluchtermann \cite{LSc}. For completeness, we give at the end of the paper a short proof of Theorem  \ref{L4}, unifying arguments from  \cite{LSc} and \cite{Lom4} (and correcting some inaccuracy in \cite{Lom4}) --- in this form it was not published before.

\section{Main results }

     In this section $X$  is a real Banach space (complex spaces are considered as real ones). The standard epimorphism from $\mathcal{B}(X)$ to the Calkin algebra   $\mathcal{C}(X) = \mathcal{B}(X)/\mathcal{K}(X)$ is denoted by $\pi$.

   Let us say that an element $a$  of a unital normed algebra is {\it positive}, if $\|1-\varepsilon a\|\le 1+o(\varepsilon)$  for $\varepsilon > 0$, $\varepsilon\to 0$. And let us say that an element $a$ is {\it real}, if $a^2$ is positive. An operator $T\in \mathcal{B}(X)$ is {\it essentially real}, if $\pi(T)$ is a real element of $\mathcal{C}(X)$.

  Clearly all selfadjoint operators in a Hilbert space are real. It is not difficult to check that Hermitian operators in a complex Banach space (defined by the condition $\|\exp(itT)\| = 1$ for $t\in \mathbb{R}$) are real. Many other operators, for instance, all involutions and all nilpotents of index two are also real. So we can see that the class of essentially real operators is very wide.

\begin{theorem}\label{LB} If $A\subset \mathcal{B}(X)$ is a commutative algebra of essentially real operators, then there exists a non-trivial closed subspace of $X^*$, invariant for the algebra $A^*= \{T^*: T \in A\}$.
\end{theorem}
\begin{proof} Note that the set of all positive elements of a Banach algebra is a convex cone. Moreover this cone is norm-closed. Indeed let $a = \lim_{n\to \infty}a_n$ where all $a_n$ are positive. If  $a$ is not positive then there is a sequence $\varepsilon_n \to 0$  and a number   $C> 0$, such that $\|1 - \varepsilon_n a\| > 1 + C\varepsilon_n$ for all  $n$. Taking $k$ with  $\|a-a_k\| < C/2$, we get that $\|1 - \varepsilon_na_k\| > 1 + C\varepsilon_n - \|a-a_k\|\varepsilon_n > 1 + (C/2)\varepsilon_n$, which is a contradiction to positivity of  $a_k$.

  Hence  the set of real elements is also norm closed. Applying this to $\mathcal{C}(X)$ we see that the set of essentially real operators is norm-closed as well. This allows us to assume that the algebra $A$ is norm-closed. Obviously we may assume also that $A$ is unital. Therefore  $\exp(T)\in A$, for each  $T\in A$.

 Since $A$ is commutative it is not WOT-dense in $\mathcal{B}(X)$. Applying Theorem \ref{L4}, we find non-zero vectors $x\in X^{**}, y\in X^*$, such that the condition (\ref{ineq2}) holds.

Therefore, for $T\in A$ and $\varepsilon \searrow 0$, we have
    $$(x,(1-\varepsilon (T^2)^*)y) \le \|1-\varepsilon T^2\|_e(x,y) \le (1+o(\varepsilon))(x,y),$$
hence \;\;\;$-\varepsilon(x, (T^2)^*y)  \le o(\varepsilon)(x,y)$ and $(x,(T^2)^*y)\ge 0$, because $(x,y)\ge 0$.

Since $\exp(T) = (\exp(T/2))^2$, we get that
\begin{equation}\label{1}
 (x,\exp(T^*)y) \ge 0 \text{ for }T\in A.
 \end{equation}
 Let $K$ be the closed convex envelope of the set $M = \{\exp(T^*)y: T\in A\}\subset X^*$. Since $M$ is invariant under all operators $\exp(T^*)$, the same is true for $K$.
 
 Since $K\neq X^*$ (by (\ref{1})) and $K$ is not a singleton (otherwise we trivially have an invariant subspace $\mathbb{C}y$), the boundary $\partial K$ of $K$ is not a singleton.
 By the Bishop-Phelps theorem \cite{BPh}, the set of support points is dense in $\partial K$, so there is a non-zero functional $x_0\in X^{**}$ supporting $K$ in some point $0\neq y_0\in K$. That is
 $$(x_0,y_0) \le (x_0,z) \text{ for all } z\in K.$$
 By the above arguments, $\exp(T^*)y_0\in K$ for $T\in A$, therefore $$ (x_0,y_0) \le (x_0,\exp(T^*)y_0) \text{ for all } T\in A.$$
 It follows that for each $T\in A$, the function $\phi(t) = (x_0,\exp(tT^*)y_0)$ has a minimum at $t = 0$. For this reason $\phi'(0) = 0$ and
 $(x_0,T^*y_0) = 0$. Hence the subspace $A^*y_0$ is not dense in $H$, and its norm-closure $\overline{A^*y}$ is a non-trivial invariant subspace.
\end{proof}

\begin{corollary}\label{LBref} A commutative algebra of essentially real operators in a reflexive real Banach space has a non-trivial invariant subspace.
\end{corollary}
Since the algebra generated by a commutative family of essentially selfadjoint operators consists of essentially selfadjoint operators we get the following result:

\begin{corollary}\label{L} Any commutative family of essentially selfadjoint operators in a real Hilbert space has a non-trivial invariant subspace.
\end{corollary}


Returning to individual criteria of continuity let us consider the class $(E)$ of operators $T$ such that all polynomials $p(T)$ of $T$ are essentially real.

\begin{corollary}\label{LB2} Each operator $T\in (E)$ in a reflexive real Banach space has a non-trivial invariant subspace.
\end{corollary}

Atzmon, Godefroy and Kalton \cite{AGK}  introduced the class $(S)$ of all operators, satisfying the condition
\begin{equation}\label{S}
\|p(T)\|_e \le \sup\{|p(t)|: t\in L\}, \text{ for each polynomial } p,
\end{equation}
 where $L$ is a compact subset of $\mathbb{R}$. It was proved in \cite{AGK} that all operators in $(S)$ have invariant subspaces. It is not difficult to see that $(S)\subset (E)$ (indeed if $T\in (S)$ then $\|1- \varepsilon p(T)^2\|_e \le \sup \{ |1-\varepsilon p(t)^2|: t\in L\} \le 1$ if $\varepsilon$ is sufficiently small) , so this result follows from Corollary \ref{LB2}.

\section{Proof of Theorem \ref{L4}}

 Without loss of generality one can assume that  $A$ is norm-closed. Since the algebra $A$   is not WOT-dense in $\mathcal{B}(X)$, the algebra $A^*$   is not WOT-dense in $\mathcal{B}(X^*)$. Suppose, aiming at the contrary, that $A^*$ is transitive (has no invariant subspaces).  Set $F = \{T^*\in A^*: \|T\|_e < 1\}$ and fix $\varepsilon \in (0, \frac{1}{10})$.  Choose $y_0\in X^*$ with $\|y_0\| = 3$ and let $S$ be the ball $\{y\in X^*: \|y-y_0\| \le 2\}$.

Let us suppose  firstly that $Fy$ is dense in $X^*$ for each non-zero $y\in X^*$. Then the same is true for $\varepsilon Fy$. It follows that for every $y\in S$  there exists $T^*_y\in \varepsilon F$ with $\|T^*_yy - y_0\|<1$. By the definition of  $F$,  $T^*_y = R^*_y + K^*_y$, where $\|R_y\| < \varepsilon$, $K_y \in \mathcal{K}(X)$. Thus $\|K^*_yy - y_0\| \le \|T^*_yy-y_0\| + \|R^*_yy\| <  1 + 5\varepsilon$ (since $\|y\| \leq 5$, for each $y\in S$). Let $\tau$ denote the (relative) *-weak topology  of $S$, then compactness of  $K_y$ implies that $K_y^*$ continuously maps $(S,\tau)$ to $(X^*, \|\cdot\|)$. Therefore, for each $y\in S$, there is a $\tau$-neighborhood  $V_y$  of  $y$   such that $\|K^*_yz - y_0\| < 1 +5\varepsilon$, for each $z\in V_y$, and $\|T^*_yz-y_0\|< 1+5\varepsilon + 5\varepsilon < 2$. In other words $T^*_y$ maps $V_y$ to $S$.

The sets $V_y$, $y\in S$, form an open covering of   $S$. Since $S$ is $\tau$-compact there is a finite subcovering $\{V_{y_i}: 1\le i\le n\}$. Let $\{\varphi_i: 1\le i\le n\}$ be a partition of unity related to this subcovering. We define a $\tau$-continuous map $\Phi: S \to S$ by $\Phi(y) = \sum_{i=1}^n\varphi_i(y)T^*_{y_i}(y)$. By Tichonov's Theorem, $\Phi$ has a fixed point $z\in S$. This means that $T^*z = z$, where $T^* = \sum_{i=1}^n\varphi_i(z)T^*_{y_i}$. Since the set $\varepsilon F$ is convex we get that  $T^* \in \varepsilon F$ and $\|T^*\|_e \leq \|T\|_e < \frac{1}{10}$.

 For this reason 1 is an eigenvalue of $T^*$ exceeding $\|T^*\|_e$ and hence it is an isolated point in $\sigma(T^*)$.  The corresponding Riesz projection is of finite rank and belongs to $A^*$. But it is well known (see e.g. \cite[Theorem 8.2]{RR}), that a transitive algebra containing a non-zero finite rank operator is $WOT$-dense in the algebra of all operators. Since $A^*$ is not $WOT$-dense in $\mathcal{B}(X^*)$, we obtain a contradiction.

  Hence there exists $y_0\in X^*$ such that $Fy_0$ is not norm-dense in  $X^*$. Let $V$ be the norm-closure of $Fy_0$.  If  $V = \{0\}$ then we have the inequality (\ref{ineq2}) with $y=y_0$ and any non-zero $x$.

  If  $V \neq \{0\}$ then $V$ is a norm-closed convex proper subset of $X^*$ containing more then one point. By the Bishop - Phelps Theorem  \cite{BPh}, there are  $0\neq y \in V$ and $0\neq x \in X^{**}$ such that  $Re(x,y) = \sup \{Re(x,w): w\in V\}$. Since the set $V$ is invariant under multiplication by any  number $t$ with $|t| \le 1$, we have $(x,y) \geq 0$ and $(x,y) \geq |(x,w)|$ for all $ w\in V$.  Since $F^2\subset F$, we have that $Fy\subset V$ and $|(x,T^*y)| \le (x,y)$, for any $T^*\in F$. Therefore  the inequality  $|(x,T^*y)|\le \|T\|_e(x,y)$ holds for all $T\in A$. $~~\blacksquare$

\medskip

\medskip

\noindent Dept of Math.\\
\noindent Kent State University\\
\noindent Kent OH 44242, USA \\
lomonoso@math.kent.edu

\medskip

\noindent Dept of Math.\\
\noindent Vologda State University\\
\noindent Vologda 160000, Russia\\
shulman.victor80@gmail.com

\end{document}